\documentclass{amsart}
\usepackage{amssymb,latexsym,amsmath,amscd,epsfig,amsthm, amsfonts}

\newtheorem{theorem}{Theorem}
\newtheorem{lemma}{Lemma}     
\newtheorem{corollary}{Corollary}
\newtheorem{proposition}{Proposition}
\theoremstyle{remark}
\newtheorem*{remark}{Remark}
\newtheorem{definition}{Definition}
\newtheorem*{example}{Example}

\newcommand{\Z}{{\mathbf Z}}
\newcommand{\T}{{\mathbf T}}

\newcommand{\N}{{\mathbf N}}
\newcommand{\R}{{\mathbf R}}
\newcommand{\C}{{\mathbf C}}
\newcommand{\m}{{m^*}}

\newcommand{\supp}{{\rm supp}}

\newcommand{\cE}{{\mathcal E}}

\newcommand{\A}{{\mathcal A}}

\newcommand{\F}{{\mathcal F}}

\newcommand{\beq}{\begin{equation}}
\newcommand{\eeq}{\end{equation}}
\newcommand{\beqs}{\begin{equation*}}
\newcommand{\eeqs}{\end{equation*}}
\newcommand{\beg}{\begin{gather}}
\newcommand{\eeg}{\end{gather}}

\begin{document}
\title[Wavelet characterization of growth spaces of harmonic functions]%
 {Wavelet characterization of growth spaces of harmonic functions}%
\author{Kjersti Solberg Eikrem}
\address{Department of Mathematical Sciences, Norwegian University of Science and Technology, NO-7491, Trondheim, Norway}
\email{kjerstei@math.ntnu.no} 

\author {Eugenia Malinnikova}
 \address{Department of Mathematical Sciences, Norwegian University of Science and Technology, NO-7491, Trondheim, Norway}
\email{eugenia@math.ntnu.no} 

\author{Pavel A. Mozolyako}
\address{Department of Mathematics and Mechanics and Chebyshev Laboratory, St.Petersburg State University, 198904, St.Petersburg, Russia}
\email{pmzlcroak@gmail.com}

\subjclass[2010]{31B25, 42C40, 60G46.}
\keywords{Growth spaces of harmonic functions, boundary behavior, wavelets, multiresolution approximation,  martingales, law of the iterated logarithm}
\begin{abstract}
We consider the space $h_v^\infty$ of  harmonic functions in $\R^{n+1}_+$ with finite norm  $\|u\|_v=\sup|u(x,t)|/v(t)$, where the weight $v$ satisfies the doubling condition. Boundary values of functions in $h_v^\infty$ are characterized in terms of their smooth multiresolution approximations. The characterization yields the isomorphism of Banach spaces $h_v^\infty\sim l^\infty$. The results are also applied to obtain the law of the iterated logarithm for the oscillation of  functions in $h_v^\infty$ along vertical lines.       
\end{abstract}

\maketitle

\section{Introduction}

\subsection{Growth spaces of harmonic functions}

Let $v:\R_+\rightarrow\R_+$ be a continuous decreasing function, $\lim_{t\rightarrow 0+} v(t)=+\infty,$ $v(t)=1$ when $t>1$, that satisfies the doubling condition
\begin{equation}
\label{eq:double}
v(t)\le Dv(2t).
\end{equation}

We consider harmonic functions in $\R^{n+1}_+$ with the following growth restriction
\[
|u(x,t)|\le Kv(t),\quad {\text{where}}\ (x,t)\in\R_+^{n+1}.\]
The space of these functions is denoted by $h_v^\infty(\R^{n+1})$ and the least  $K$ for which the inequality above is satisfied is called the norm of $u$ in $h_v^\infty$, we denote it by $\|u\|_{v,\infty}$.
We note that such a harmonic function is bounded in any half-space 
\begin{equation}\label{eq:half}
\R^{n+1}_{\delta}=\{(x,t)\in\R^{n+1}, t\ge \delta>0\}
\end{equation} 
and thus can be represented there by the Poisson integral of its values on the hyperplane $\{t=\delta\}$. We denote by $h_v^0$ the subspace of $h_v^\infty$ consisting of functions $u$ such that $u(x,t)=o(v(t))$ $(t\rightarrow 0)$ uniformly in $x\in\R^n$.
  
Similar growth spaces of analytic and harmonic functions in the unit disk were considered by A.~L.~Shields and D.~L.~Williams in \cite{SW1} and \cite{SW}. The Fourier series of such functions were studied in \cite{BST}, where in particular it is proved that the growth of the function cannot be characterized by the growth of the partial sums of its Fourier series, but is described by the growth of the Ces\`{a}ro means of the Fourier series. In the last decades a thorough study of the isomorphism classes of such weighted spaces was done by W. Lusky, see \cite{L} and the references therein.  For some explicit weights the growth spaces of analytic and harmonic functions in the unit disk and unit ball in $\C^n$  have been intensively studied. We mention here the classical article by B.~Korenblum \cite{K} that was a starting point for interesting research in the area, and a recent article by K.~Seip \cite{Se}, where weighted Hilbert spaces of analytic functions with slow growing weights were considered.  


\subsection{Formulation of the main results}
In the present work we give a description of (the boundary distribution of) functions in the growth spaces in terms of their multiresolution approximation.  Wavelet series of the boundary values of a harmonic function is a convenient tool that replaces the Fourier series. For the simplest case of Haar wavelets, we get the standard martingale representation. For our purposes we choose smooth multiresolution analysis, the smoothness depends on the weight $v$. This allows us to describe the space $h_v^\infty$ and understand its (Banach space)  geometry and also to obtain smooth approximations of individual functions in  $h_v^\infty$. Initially we were interested in the boundary behavior of harmonic functions and were looking for a smooth version of a  martingale decomposition, see Section \ref{s:osc} for details, however we believe that the wavelet characterization is interesting in itself. 

When the weight grows faster than $t^{-a}$ for some $a$, our description is in terms of the wavelet coefficients; for slow growing weights we consider partial sums of the wavelet series. Description of various functional spaces in terms of their wavelet (or atomic) decomposition is a classical topic in analysis, atomic characterization of many spaces are known, many of these results are translated onto the language of multiresolution approximations and wavelets, we refer the reader to \cite[Ch. 6]{Mey}. In the present work we describe spaces of distributions with some "smoothness"  that depends on the weight $v$.   

Basic notions and results of the wavelet theory are collected in the next section, here we introduce some notation to formulate the main results.
Let $v$ be a weight as above that satisfies the doubling condition. We consider an $r$-regular multiresolution approximation $\{V_j\}_j$ of $L^2(\R^n)$, \cite[ch 2.2]{Mey}, where $r\ge r_0(v)$ will be specified later. Then there exists  $\phi\in V_0$ that satisfies
\[
|\partial^\alpha\phi(x)|\le C_N(1+|x|)^{-N},\]
for any $\alpha$ such that $|\alpha|\le r$ and every $N\in\N$, and $\{\phi(x-k), k\in\Z^n\}$ form an orthonormal basis for $V_0$. Further, there exists 
a collection of smooth (of class $C^r$) functions $\{\psi_p\}_{p=1}^q$ that form an orthonormal basis for $V_1\ominus V_0$, decrease rapidly with all its derivatives of order up to $r$ and satisfy the cancellation property. Then 
\[\psi_{p,jk}=2^{nj/2}\psi_p(2^jx-k),\quad j\in \Z, k\in \Z^n,\ p=1,...,q,\]
is an orthonormal wavelet basis in $L^2(\R^n)$, \cite[ch 3.6]{Mey}. We will use the orthonormal basis $\{\phi(x-k)\}_{k\in\Z^n}\cup\{\psi_{p,jk}\}_{1\le p\le q, j\ge 0, k\in\Z^n}$. For any function $f\in L^\infty(\R^n)$ we define
 \[
c_{p,jk}(f)=\int_{\R^n}f(x)\overline{\psi_{p,jk}(x)}dx,\quad j\ge 0, k\in\Z^n,\ p=1,...,q,\]  
and
\[
b_k(f)=\int_{\R^n}f(x)\overline{\phi(x-k)}dx,\quad k\in\Z^n.\]
Finally, the partial sum of the wavelet decomposition of $f$ is denoted by
\[
s_N(f)(x)=\sum_k b_k(f)\phi(x-k)+\sum_{p=1}^q\sum_{j=0}^{N}\sum_k c_{p,jk}\psi_{p,jk}(x).\]
Now we can formulate our main result.

\begin{theorem}\label{th:intr}
Let $u(x,t)$ be a harmonic function on $\R^{n+1}_+$ bounded on each half-space $\{(x,t): t>t_0>0\}$. 
Then $u\in h_v^\infty$ if and only if there exists $K$ such that \[M_N(u)=\sup_{t>0}\|s_N(u(\cdot,t))\|_{L^\infty(\R^n)}\le Kv(2^{-N}).\]
Similarly, $u\in h_v^0$ if and only if $\lim_{N\rightarrow \infty}M_N(u)(v(2^{-N}))^{-1}=0$.
\end{theorem} 

The proof of the theorem above combines standard tools of multiresolution analysis with a clever argument of J.~Bourgain, \cite{Bour}, that allows one to squeeze a convolution with the appropriate Poisson kernel. This trick was also used in \cite{O,Moz}. 

The result yields in particular an isomorphism $h_v^\infty\sim l^\infty$  for weights with the doubling property. For the case of the unit disk the isomorphism was obtained by W.~Lusky, \cite{L0, L00} by different methods. We also refer the reader to \cite{L} for recent results on isomorphic classes of growth spaces of holomorphic and harmonic functions on complex disk and plane, where non-doubling weights are considered.

In the course of the proof of Theorem \ref{th:intr}  we obtain another characterization of functions in $h_v^\infty$ that does not refer to multiresolution analysis. 
Let $g$ be a non-zero radial function in $\R^n$ such that $g\in C^r$, where $r$ is large enough ($r>r_0(v)$). Assume also that $g$ with all its partial derivatives of order up to $r$ satisfies
\[
|\partial^\beta g(x)|\le \frac{C}{(1+|x|^2)^{n+1}}.\]
For example $g$ with compact support will work. Then $(1+|x|^{n+1})\partial^\beta g\in L^1(\R^n)$ when $|\beta|\le r$. We have
\begin{equation}\label{eq:gFour}
|\hat{g}(\tau)|\le \frac{C}{(1+|\tau|)^r},
\end{equation}
and similar estimates hold for partial derivatives of $\hat{g}$ up to order $n+1$.  


\begin{theorem}\label{th:intr1}
Let $u$ be a harmonic function on $\R^{n+1}_+$ that is bounded in each half-space $\{(x,t): t\ge t_0>0\}$ and let $g\in L^1(\R^n)$ be a radial function such that $\hat{g}$ has derivatives in $L^1(\R^n)$ up to order $n+1$, (\ref{eq:gFour}) holds for $\hat{g}$ and its derivatives, and $\hat{g}(0)\neq 0$. Then $u\in h^\infty_v$ if and only if there exists a constant $C_u$ such that
\begin{equation}\label{eq:WT}
\left|\int_{\R^n} u(x,t) g\left(\frac{y-x}{a}\right)dx\right|\le C_u a^{n}v(a),
\end{equation}
for all $t>0, a>0$ and $y\in\R^n$. 
\end{theorem} 

An example of such function $g$ is given by $g(x)=h(|x|)$, where $h=\chi\ast\chi\ast\cdots\ast\chi$ and $\chi$ is the characteristic function of the interval $[-1/2,1/2]$. The number of factors equals $r$. Then $g$ has compact support and satisfies the conditions of the theorem. 

The formulation of the theorem is inspired by the Korenblum's premeasures of bounded $\kappa$-variations, \cite[Section 5]{K}. We remind that a harmonic function $U(z)$ in the unit disk satisfies 
\[
U(z)\le C\log\frac{e}{1-|z|},\ U(0)=0\]
if and only if $U$ is the Poisson integral of a finitely additive function $\mu$ (called a premeasure) defined on finite unions of subintervals of the unit circle such that  $\mu(I_n)\rightarrow 0$ for a sequence $I_1\supset I_2\supset...$ of intervals that satisfies $\cap_n I_n=\emptyset$, $\mu(\T)=0$ and $\mu(I)\le C|I|\log|I|$, where $I$ is an interval on the circle and $|I|$ is its normalized Lebesgue measure. This is equivalent to the following estimate
\[
\int_{I}U(re^{i\theta})d\theta\le C|I|\log\frac{e}{|I|}
\]
for any interval $I$ and any $r\in(0,1)$, which resembles (\ref{eq:WT}).  Let us also make clear that our result is applicable only to harmonic functions with estimates of the absolute value (two-sided estimates), while the Korenblum estimate above holds for a substantially wider class of harmonic functions bounded from above (one-sided estimate).


\subsection{Boundary behavior of functions in the growth spaces}  
In the second part of the article we use the wavelet description of functions in $h_v^\infty$ to study their boundary behavior. Previous results on the boundary limits of functions with such growth conditions can be found in \cite{BLMT, LM, EM, E}. We consider the weighted averages of $u$ along vertical lines:
\[
I_u(x,s)=\int_s^1 u(x,t)d\left(\frac1{v}(t)\right).\]
Clearly $|I_u(x,s)|\le K\log v(s).$ We want to show that 
\beq\label{eq:I}
\limsup_{s\rightarrow 0+}\frac{I_u(x,s)}{\sqrt{\log v(s)\log\log\log v(s)}}\le C,
\eeq
for almost every $x\in\R^n$. Similar results were obtained in the unit disk for the weight $v(t)=\log1/t$ in \cite{LM} and for Hadamard gap series in general growth spaces in \cite{E}. Our calculations show that the scheme developed in \cite{LM} works for slow-growing weights. In order to include weights that grow faster but satisfy the doubling condition (for example weights that grow at zero faster than powers of $\log 1/t$ but slower than powers of $t^{-1}$), we use multiresolution approximation and orthogonal wavelets with compact supports. Further, as in \cite{LM}, we construct a martingale approximation of $I_u$. Finally the law of the iterated logarithm is applied to obtain (\ref{eq:I}). Martingales and laws of iterated logarithm are by now classical tools to study the boundary behavior of harmonic functions, see \cite{M, CWW} for nice applications of this technique and \cite{BKM,BM} for detailed expositions of the topic, we also refer the reader to \cite{LN} for interesting laws of the iterated logarithm in real analysis.  The main difficulty in our case is to construct a suitable martingale and give the estimates of its square function. We work with dyadic $\sigma$-algebras but take approximations only on some levels, depending on the growth of the weight $v$. Moreover, to obtain the estimates we cut off appropriate high frequencies on each level. In other words we divide our wavelet decomposition into blocks, take averages on the corresponding scales, estimate the resulting martingale and the approximation error.   

The article is organized as follows. In the next section we prove the main lemma (based on the ideas of J.~Bourgain, \cite{Bour}) and Theorem \ref{th:intr1}, we also collect necessary results on multiresolution analysis (following \cite{Mey}). In section \ref{s:main} we  prove Theorem \ref{th:intr}; we work with blocks of wavelets of consecutive generations, where the size of the block depends on the weight function $v$, we also discuss the isomorphism $h_v^\infty\sim l^\infty$. Section \ref{s:osc} is devoted to the estimate (\ref{eq:I}). Finally, in the last section we collect some remarks and open problems.

\section{Preliminaries}
\subsection{Auxiliary results and Main lemma} In this subsection we prove some lemmas. The first one is an elementary estimate.
\begin{lemma}\label{l:0}
Let $g\in L^1(\R^n)$ and $|x|^{2M}g(x)\in L^1(\R^n)$ for some integer $M>n/2$, if $\|\hat{g}\|_1=C_1$ and 
$\|\Delta^{M}(\hat{g})\|_1\le C_2$ then 
\[
\|g\|_1\le c(n,M)\left(C_1^{2M-n}C_2^n\right)^{1/(2M)}.\] 
\end{lemma}

\begin{proof}
Since $g,\hat{g}\in L^1(\R^n)$ the inversion formula implies that $|g(x)|\le C_1$. Similarly, $|x|^{2M}g(x)  , \widehat{|x|^{2M}g(x)}\in L^1$ and $|g(x)|\le C_2|x|^{-2M}$. Thus
\[
\|g\|_1\le c_n\left(C_1\frac{R^n}{n}+C_2\frac{R^{n-2M}}{2M-n}\right),
\]
for any $R>0$. Choosing $R$ such that $R^{2M}=C_2C_1^{-1}$ we obtain the required estimate.
\end{proof}

Let $P(x)=c_n(1+|x|^2)^{-(n+1)/2}$ be the standard Poisson kernel, $P_s(x)=s^{-n}P(\frac{x}{s})$ as usual. The constant $c_n$ is chosen such that  $\hat{P}(\tau)=e^{-2\pi|\tau|}$.
The next lemma will be used to divide by the Poisson kernel in the Fourier transforms and is inspired by \cite{Bour, O}. We give a proof for the convenience of the reader.
\begin{lemma}\label{l:1}
There exists $\Sigma\in L^1(\R^n)$ such that $\hat{\Sigma}(\tau)=e^{2\pi|\tau|}$ when $|\tau|\le 1$.
\end{lemma}

\begin{proof}
We note that $2\cosh(2\pi|\tau|)$ is a smooth function in $\R^n$ and we can find a smooth function $\Theta$ with compact support such that $\Theta(\tau)=2\cosh(2\pi|\tau|)$ when $|\tau|\le 1$. Let $\Xi$  be the inverse Fourier transform of $\Theta$. Clearly, $\Xi\in L^1(\R^n)$. Finally we let $\Sigma=\Xi-P$, then $\Sigma\in L^1(\R^n)$ and $\hat{\Sigma}(\tau)=\Theta(\tau)-e^{-2\pi|\tau|}$.  For $|\tau|\le 1$ we have $\hat{\Sigma}(\tau)=e^{2\pi|\tau|}$.
\end{proof}

Now we can give a preliminary estimate for a part of $u(\cdot,t)\in h_v^\infty$  with bounded frequencies. The next result is our Main lemma.

\begin{lemma}\label{l:2}
Let $u$ be a  bounded harmonic function in $\R^{n+1}_+$, and let $\sigma\in L^1(\R^n)$ be such that $\supp\,\hat{\sigma}\subset B_{\delta^{-1}}$. Then 
\[
\left|\int_{\R^n}u(x,t)\sigma(x)dx\right|\le C_n \|u\|_{L^\infty(\R^{n+1}_\delta)} \|\sigma\|_1,
\]
where $C_n$ is  a constant that depends on $n$ only and $\R^{n+1}_\delta$ is defined by (\ref{eq:half}). 
\end{lemma}

\begin{proof}
First, in the sense of distributions, we have
\[
\int_{\R^n}u(x,t)\sigma(x)dx=\int_{\R^n}\widehat{u(\cdot,t)}(\tau)\hat{\sigma}(\tau)d\tau.\]
Now let $\Sigma$ be the function in Lemma \ref{l:1} and let $\Sigma_\delta(x)=\delta^{-n}\Sigma(\delta^{-1}x)$, then $\widehat{\Sigma_\delta}(\tau)=\hat{\Sigma}(\delta\tau)$ and $\|\Sigma_\delta\|_1=\|\Sigma\|_1$. Since $\hat{\sigma}$ vanishes outside the ball $B_{1/\delta}$, we have 
\[\hat{\sigma}(\tau)=\widehat{\Sigma_\delta}(\tau)e^{-2\pi|\tau|\delta}\hat{\sigma}(\tau)=\widehat{\Sigma_\delta}(\tau)\hat {P_{\delta}}(\tau)\hat{\sigma}(\tau).\]
Then
\begin{multline*}
\left|\int_{\R^n}u(x,t)\sigma(x)dx\right|=\left|\int_{\R^n}\widehat{u(\cdot,t+\delta)}(\tau)\widehat{\sigma\ast\Sigma_\delta}(\tau)d\tau\right|=\\
\left|\int_{\R^n}u(x,t+\delta)(\sigma\ast\Sigma_\delta)(x)dx\right|\le
\|u\|_{L^\infty(\R^{n+1}_\delta)}\|\sigma\|_1\|\Sigma_\delta\|_1
\end{multline*}
and the required estimate follows.  
\end{proof}


\subsection{Proof of Theorem \ref{th:intr1}}
First, we prove the following.

\begin{proposition}\label{cor:0}
Let $u(x,t)$ be a harmonic function in $\R^n_+$ bounded on each half-space $\{(x,t), t>t_0>0\}$. Then  $u\in h_v^\infty(\R^{n+1}_+)$ if and only if there exists a constant $C$ such that 
\begin{equation}\label{eq:n1}
\left|\int_{\R^n}u(x,t)\sigma(x)dx\right|\le C v(\delta)\|\sigma\|_1,
\end{equation}
for any $t>0$ and any $\sigma\in L^1 (\R^n)$  such that $\supp\,\hat{\sigma}\subset B_{\delta^{-1}}$.
Similarly,  $u\in h_v^0(\R^{n+1}_+)$ if and only if for any $\epsilon>0$ there exists $\delta(\epsilon)$ such that if $\delta<\delta(\epsilon)$  and $\sigma\in L^1(\R^n)$ satisfies $\supp\,\hat{\sigma}\subset B_{\delta^{-1}}$, then for any $t>0$
\[
\left|\int_{\R^n}u(x,t)\sigma(x)dx\right|\le \epsilon v(\delta)\|\sigma\|_1.
\] 
\end{proposition}

\begin{proof} If $u\in h_v^\infty$ then Main lemma implies 
\[
\left|\int_{\R^n}u(x,t)\sigma(x)dx\right|\le C_n \|u\|_{v,\infty} v(\delta)\|\sigma\|_1,
\]
where $C_n$ depends only on the dimension of the space.
Conversely, assume that $u$ is a harmonic function satisfying (\ref{eq:n1}).  Let further $\eta:\R\rightarrow\R$ be a smooth function with  support in $[-1,1]$ that is equal to $1$ on $[-1/2,1/2]$, and $0\le\eta\le 1$. We have 
\[
P(x)=\F^{-1}(e^{-2\pi|\tau|}\eta(|\tau|))+\sum_{j=1}^{\infty}\F^{-1}\left(e^{-2\pi |\tau|}(\eta(2^{-j}|\tau|)-\eta(2^{-j+1}|\tau|))\right)=\sum_{j=0}^\infty\sigma_j.\]
Each term $\sigma_j$ satisfies the condition of the proposition with $\delta_j=2^{-j}$  and has $\|\sigma_j\|_1\le C(\eta)2^{jn}\exp(-2^{j-1}\pi)$.
Then $P_s(x)=\sum_{j=0}^\infty (\sigma_j)_s(x)$ for any $s>0$. 
For any $t>0$ we have
\begin{multline*}
|u(x,s+t)|= |(u(\cdot,t)\ast P_s)(x)|\le\\
 \sum_{j=0}^\infty Cv(2^{-j}s)2^{jn}\exp(-2^{j-1}\pi)\le Cv(s)\sum_{j=0}^\infty (2^nD)^j\exp(-2^{j-1}\pi).
\end{multline*}
The last series converges and the first statement of the proposition follows. The second can be obtained in a similar way. 
\end{proof}

To finish the proof of Theorem \ref{th:intr1} we want to show that (\ref{eq:WT}) is equivalent to (\ref{eq:n1}). The argument repeats the one we gave above for the Poisson kernel. First, we have to divide by the Fourier transform of $g$ as in Main lemma. It is more simple for $g$ than for the Poisson kernel. Remind that $\hat{g}(0)\neq 0$, $g(\tau)=g(|\tau|)$ and since $\hat{g}$ is smooth enough there is exists a ball of radius $\rho>0$ and a function $\Sigma_g$ such that $\widehat{\Sigma_g}\hat{g}=1$ in $B(\rho)$. We note also that by the doubling condition $v(\rho^{-1}\delta)\le C(\rho)v(\delta)$ and then 
 (\ref{eq:WT}) implies (\ref{eq:n1}). To prove the converse we follow the argument in the last proposition. The decay of the Fourier transform of $g$ is in general not as fast as that of the Fourier transform of the Poisson kernel. However, applying (\ref{eq:gFour}) for $\hat{g}$ and its derivatives and $r>r_0$ we get the required estimate. Indeed, the series $\sum_j (2^nD)^j2^{-rj}$ converges for $r$ large enough.


\subsection{Basic facts about smooth multiresolution analysis}
We consider smooth (of order $r$) multiresolution approximation (MRA) in $\R^n$.
Our main references here are the classical books by I.~Daubechies \cite{D} and by Y.~Meyer  \cite{Mey}. We denote by $E(x,y)$ the kernel of the orthogonal projection onto $V_0$ and assume that  
\[
E(x,y)=\sum_{k\in \Z^n}\phi(x-k)\overline{\phi(y-k)},\] 
where $\phi(x)$ and all its derivatives of order up to $r$ decay  faster at infinity than any power of $x$. We assume further that  
\[
\sum_{k\in\Z^n}\phi(x-k)=1,\]
see \cite[ch 2.10]{Mey}. 

Let  $E_j(x,y)=2^{jn}E(2^jx,2^jy)$. Further, let $D(x,y)=E_{1}(x,y)-E(x,y)$ and \[D_j(x,y)=2^{jn}D(2^jx,2^jy)=E_{j+1}(x,y)-E_j(x,y).\]
Since we work with $r$-smooth MRA, we have
\beq\label{eq:smoothMRA}
 D(x,y)=\sum_{|\beta|=r}\partial^\beta_y D_\beta(x,y) ,
\eeq
where $D_\beta$ are Schwartz functions that satisfy $|D_\beta(x,y)|\le C_m(1+|x-y|)^{-m}$, see \cite[ch 2.8]{Mey}. Further for any $f\in L^2(\R^n)$ we define 
\[
E_0f(x)=\int_{\R^n} E(x,y)f(y)dy\quad {\text{and}}\quad D_jf(x)=\int_{\R^n} D_j(x,y)f(y)dy.\]

As usual $V_j=\{f(x): f(2^{-j}x)\in V_0\}$. We will also need $L^\infty$-version of these spaces,
\[
V_0(\infty)=\{f(x)=\sum_{k\in\Z^n}a(k)\phi(x-k),\quad \{a(k)\}\in l^{\infty}(\Z^n)\},\]
and $V_j(\infty)=\{f(x): f(2^{-j}x)\in V_0(\infty)\}$, $V_j\subset L^\infty(\R^n)$. 

The following Bernstein's inequality holds, \cite[ch 2.5]{Mey}. There exists $C=C(\phi)$ such that
\beq\label{eq:Bern}
\|\partial^\beta f\|_\infty\le C2^{|\beta|j}\|f\|_\infty
\eeq
for any $f\in V_j(\infty)$ and any multi-index $\beta$ such that $|\beta|\le r$.



\subsection{Multiresolution approximation of the Poisson kernel}
We need two estimates for smooth multiresolution approximation of Poisson kernels. 

\begin{lemma}\label{l:PE}
There exists $C$ such that 
\[
\int_{\R^n}\left|\int_{\R^n}(P_s(w-y)-P_s(w-x))E(x,y)dx\right| dy\le Cs^{-r}
\]
for any $w\in\R^n$.
\end{lemma}

\begin{proof}
We denote $f_w(x)=P_s(w-x)$, then
\begin{multline*}
\int_{\R^n}(P_s(w-y)-P_s(w-x))E(x,y)dx=\\
P_s(w-y)-\int_{\R^n}f_w(x)E(y,x)dx=f_w(y)-(E_0f_w)(y).
\end{multline*}
 We have
\[
|f_w(x)-E_0f_w(x)|\le\sum_{j=0}^\infty\left|\int_{\R^n} D_j(x,y)f_w(y)dy\right|.\]
Integrating (\ref{eq:smoothMRA}) we obtain
\[
\int_{\R^n}\left|\int_{\R^n} D(x,y)f_w(y)dy\right|dx\le C\sum_{|\beta|=r}\|\partial^\beta f_w\|_1\le Cs^{-r}.\]
By rescaling we have also $D_jf(x)=(D_0f_j)(2^jx)$, where $f_j(y)=f(y2^{-j})$. Then 
\[
\|D_jf\|_1=2^{-nj}\|Df_j\|_1\le C2^{-rj}\sum_{|\beta|=r}\|\partial^\beta f\|_1.\]
We apply this estimate for every $j=0,1,...$ and $f=f_w$ and get
\[
\int_{\R^n}|f_w(x)-E_0f_w(x)|\le\sum_{j=0}^\infty\|D_jf_w\|_1\le C\sum_{j=0}^\infty 2^{-rj}\sum_{|\beta|=r}\|\partial^\beta f_w\|_1\le Cs^{-r}.\]
\end{proof}

\begin{lemma}\label{l:DP}
There exists $C$ such that
\[
 \int_{\R^n}\left|\int_{\R^n}(E_J(x,y)-E(x,y))P_s(y-w)dy\right|dw\le Cs^{-r}\]
 for any $x\in\R^n$ and any $J\ge 1$.
\end{lemma}

\begin{proof}
In the notation of the last lemma we have
\[
\int_{\R^n}(E_J(x,y)-E(x,y))P_s(y-w)dy=\sum_{j=0}^{J-1}D_jf_w(x)\] 
and by (\ref{eq:smoothMRA})
\[
\int_{\R^n}|D_0f_w(x)|dw\le C\sum_{|\beta|=r}\|\partial^\beta P_s\|_1\le C s^{-r}.\]
Then similarly for $j\ge 1$
\[
\int_{\R^n}|D_jf_w(x)|dw\le C2^{-rj}s^{-r}.\]
That concludes the proof of the lemma. 
\end{proof}



\section{Multiresolution analysis in growth spaces}\label{s:main}


\subsection{Decomposition into blocks and direct estimates}\label{ss:dec}
Now let $v$ be a weight function that satisfies the doubling condition, we choose $A$ large enough and define a sequence of integers $\{\alpha_l\}_l$ such that $\alpha_0=0$, $\alpha_l>\alpha_{l-1}$ and $v(2^{-\alpha_l})\in[A^l,A^{l+1})$. There exists $\m$ that depends on $v$ only that satisfies
\beq\label{eq:m}
\frac{2^{-\m\alpha_l}v(2^{-\alpha_l})}{2^{-\m\alpha_{l-1}}v(2^{-\alpha_{l-1}})}<\gamma<1.
\eeq
For weights $v$ with some regularity we can satisfy the last inequality by choosing $\m$ such that $t^{\m-1}v(t)$ is increasing.

We consider sufficiently smooth multiresolution approximation in $\R^n$, more precisely we choose $r$ such that $r>\m+n$, where $\m$ was chosen above. Instead of the usual dyadic partition $L^2(\R^n)=V_0\cup\bigcup_{j\ge 1} V_j\setminus V_{j-1}$ we work with blocks adjusted to the weight $v$, 
\[
L^2(\R^n)=V_0\cup\bigcup_{l\ge 1}V_{\alpha_l}\setminus V_{\alpha_{l-1}}.\]
We take wavelet series and combine all terms in generations $\alpha_ {l-1}+1,...,\alpha_{l}$ into one block; we work with bounded functions that do not belong to $L^2(\R^n)$ in general. Let $\{\psi_p\}_{p=1}^q$ be a collection of $r$-smooth rapidly decreasing functions such that $\{\psi_p(x-k), 1\le p\le q, k\in\Z^n\}$ form an orthogonal basis for $V_1\setminus V_0$, see \cite[ch 3.1, 3.6]{Mey} for details. For each $j\in \Z_+$ and $k\in\Z^n$
we have 
\[
\psi_{p,jk}=2^{nj/2}\psi_p(2^jx-k).\]

In what follows we write 
\[(f(y),g(y))=\int_{\R^n}f(y)\overline{g(y)}dy\]
when the integral converges. 
   
\begin{theorem}\label{th:1}
For any $u\in h_v^\infty(\R^{n+1}_+)$ we define 
\[g_0(x,t)=\sum_{k\in\Z^n} (u(y,t),\phi(y-k))\phi(x-k),\quad {\text{and}}
\]
\[
g_l(x,t)=\sum_{j=\alpha_{l-1}+1}^{\alpha_l}\sum_{p=1}^q\sum_{k\in\Z^n}(u(y,t),\psi_{p,jk}(y))\psi_{p,jk}(x),\quad  l\ge 1.\]
Then
\[
u(x,t)=\sum_{l=0}^\infty g_l(x,t), \quad 
g_l(\cdot,t)\in V_{\alpha_l}(\infty)\quad {\text{and}}\] 
\beq\label{eq:wave}
\|g_l(\cdot,t)\|_\infty\le C\|u\|_{v,\infty}v(2^{-\alpha_l}),\quad l\ge 0,
\eeq
where $C$ depends on $\phi$ and $A$ only.
\end{theorem}

\begin{proof}
Let as usual $E_j(x,y)=2^{jn}E(2^jx,2^jy)$, then
\[g_0(x,t)=\int_{\R^n}E(x,y)u(y,t)dt\quad {\text{and}}\] 
\[
g_l(x,t)=\int_{\R^n}(E_{\alpha_l}(x,y)-E_{\alpha_{l-1}}(x,y))u(y,t)dy.\]
Clearly, $g_l(\cdot,t)\in V_{\alpha_l}(\infty)$.
Moreover, for each $t$ the function $u(\cdot,t)$ is uniformly continuous, thus the series $\sum_{l} g_l(x,t)$ converges to $u(x,t)$ uniformly on $\R^n$. 

We take the Fourier transform of $E(x,\cdot)$ in second variable and denote it by $\hat{E}(x,\tau)$. (We never use the Fourier transform in the whole $\R^{2n}$.)  We note that \[|\hat{E}(x,\tau)|\le C(1+|\tau|)^{-\m-n}\] 
uniformly in $x$ since $r>\m+n$, the same inequality holds for all derivatives of $\hat{E}(x,\tau)$ in $\tau$. Let further $\eta:\R\rightarrow\R$ be a smooth function with support in $[-1,1]$ that is equal to $1$ on $[-1/2,1/2]$, $0\le\eta\le 1$.

Then the Fourier transform of $E_{N}(x,\cdot)$, where $N\le \alpha_l$, has the following partition
\begin{multline*}
\widehat{E_{N}}(x,\tau)=\hat{E}(2^{N}x, 2^{-N}\tau)=
\hat{E}(2^{N}x,2^{-N}\tau)\eta(2^{-\alpha_l}|\tau|)+\\\sum_{i=l+1}^\infty\hat{E}(2^{N}x,2^{-N}\tau)(\eta(2^{-\alpha_i}|\tau|)-\eta(2^{-\alpha_{i-1}}|\tau|))=
\zeta_{N,l}^0(x,\tau)+\sum_{i=l+1}^\infty\zeta_{N,i}(x,\tau).
\end{multline*}
First, we have
\[
\|\zeta_{N,l}^0(x,\cdot)\|_1\le\|\hat{E}(2^Nx, 2^{-N}\cdot)\|_1\le 2^{nN}\|\hat{E}(2^Nx,\cdot)\|_1.\] 
And also, since $N\le \alpha_l$, we obtain
$
\|\Delta_\tau^M\zeta_{N,l}^0(x,\cdot)\|_1\le C2^{(n-2M)N}.$
Let further, $\sigma_{N,l}^0=\F^{-1}(\zeta_{N,l}^0)$. Then Lemma \ref{l:0}
implies
$\|\sigma_{N,l}^0(x,\cdot)\|_1\le C$.

Next, using the estimates for the decay of $\hat{E}(x,\tau)$, we obtain
\begin{multline*}
\|\zeta_{N,i}(x,\cdot)\|_1=
\int_{\R^n}|\hat{E}(2^{N}x,2^{-N}\tau)(\eta(2^{-\alpha_i}|\tau|)-\eta(2^{-\alpha_{i-1}}|\tau|))|d\tau\le\\
2^{nN}\int_{2^{\alpha_{i-1}-N-1}\le|\xi|\le 2^{\alpha_i-N}}|\hat{E}(2^Nx,\xi)|d\xi\le C2^{nN+\m(N-\alpha_{i-1})}.
\end{multline*}
Similarly, since $\alpha_i>N$ and the derivatives of $\hat{E}(x,\tau)$ satisfy the same decay estimates, we get
\begin{multline*}
\|\Delta_\tau^M\zeta_{N,i}(x,\cdot)\|_1=\int_{\R^n}\left|\Delta_\tau^{M}\left(\hat{E}(2^{\alpha_l}x,2^{-\alpha_l}\tau)(\eta(2^{-\alpha_i}|\tau|)-\eta(2^{-\alpha_{i-1}}|\tau|))\right)\right|d\tau\\
\le C2^{(n-2M)N+\m(N-\alpha_{i-1})},
\end{multline*}
for any $M\ge 1$.

Further, we define   $\sigma_{N,i}(x,y)=\F^{-1}(\zeta_{N,i}(x,\cdot)),\ i>l.$
Then, applying Lemma \ref{l:0} once again, we have
$\|\sigma_{N,i}\|_1\le C2^{\m(N-\alpha_{i-1})}.$   
Finally, applying Lemma \ref{l:2} and (\ref{eq:m}), we obtain
\begin{multline}\label{eq:newN}
\left|\int_{\R^n}E_{N}(x,y)u(y,t)dy\right|
\le \\
C\|u\|_{v,\infty}\left(v(2^{-\alpha_l})+2^{\m N}\sum_{i=l+1}^\infty  2^{-\m\alpha_{i-1}}v(2^{-\alpha_{i}})\right)\le C\|u\|_{v,\infty}A^l,
 \end{multline} 
 for any $x\in\R^n$ and $N\le \alpha_l$.
 Then (\ref{eq:wave}) follows by taking $N=\alpha_l$ and $N=\alpha_{l-1}$.
\end{proof}

\begin{corollary}\label{cor:1}
Let $\{\psi_{p,jk}\}$ be an orthogonal smooth wavelet basis as above. Then there exist $C$ such that for any $u\in h_v^\infty$
\beq\label{eq:coef1}
|c_{p,jk}(u(\cdot,t))|\le C2^{-nj/2}\|u\|_{v,\infty}v(2^{-j}),
\eeq
when $t>0$, $j\in\Z_+,\ k\in\Z^n$. 
\end{corollary}

\begin{proof}
Clearly, 
\[|(u(x,t),\psi_{p,jk}(x))|=|(g_l(x,t),\psi_{p,jk}(x))|\le\|g_l(x,t)\|_\infty\|\psi_{p,jk}\|_1,
\]
where $j\in(\alpha_{l-1},\alpha_l]$. Then (\ref{eq:wave}) implies (\ref{eq:coef1}). 
\end{proof}

\subsection{Converse estimates and coefficient characterization for weights of power growth}
The converse of Theorem \ref{th:1} is also true. 

\begin{theorem}\label{th:2}
Let $u$ be a harmonic function in $\R_{+}^{n+1}$ that is bounded on each half-space
$\{(x,t)\in\R^{n+1}, t\ge t_0>0\}$. Suppose that for each $t>0$
\[
u(x,t)=\sum_{l=0}^\infty g_l(x,t),\]
where the series converges uniformly on $\R^n$,  
$g_0(\cdot,t)\in V_{0}(\infty)$,
\[g_l(x,t)=\sum_{j=\alpha_{l-1}+1}^{\alpha_l}\sum_{p=1}^q\sum_{k\in\Z^n}a_p^{(jk)}(t)\psi_{p,jk}(x),\  l\ge 1\]
and there exists $B$ such that 
\[
\|g_l(\cdot,t)\|_\infty\le Bv(2^{-\alpha_l}),\]
for any $t>0$. Then $u\in h_v^\infty$ and $\|u\|_{v,\infty}\le CB$, where $C$ depends on $A$ and $\phi$ only.
\end{theorem}

\begin{proof}
We fix $s\in (0,1]$ and take $L$ such that $s\in[2^{-\alpha_{L+1}},2^{-\alpha_{L}})$. Then
\begin{multline*}
u(x,t+s)=(u(\cdot,t)\ast P_s)(x)=\sum_{l=0}^\infty (g_l(\cdot,t)\ast P_s)(x)=\\
\sum_{l=0}^{L+1}(g_l(\cdot,t)\ast P_s)(x)+\sum_{l=L+2}^\infty(g_l(\cdot,t)\ast P_s)(x).
\end{multline*}
For each $l\le L+1$ we have 
\[|g_l(\cdot,t)\ast P_s|\le \|g_l(\cdot,t)\|_\infty\le Bv(2^{-\alpha_l}).\]
Since $(g_l(y,t),E_{\alpha_{l-1}}(x,y))=0$, for $l>L+1$ we get 
\begin{multline*}
(g_l(\cdot,t)\ast P_s)(x)=\\
\int_{\R^n}g_l(y,t)P_s(x-y)dy-\int_{\R^n}\int_{\R^n}g_l(y,t)E_{\alpha_{l-1}}(w,y)P_s(x-w)dwdy=\\
\int_{\R^n}g_l(y,t)\int_{\R^n}(P_s(x-y)-P_s(x-w))E_{\alpha_{l-1}}(w,y)dwdy.
\end{multline*}
Then by rescaling and applying Lemma \ref{l:PE}, we obtain
\[
|g_l(\cdot,t)\ast P_s|\le CBv(2^{-\alpha_{l}})(2^{\alpha_{l-1}}s)^{-r}.
\]
Now we remark that $2^{\alpha_{l-1}}s\ge 1$ and $r>\m$, where $\m$ is chosen such that (\ref{eq:m}) holds. Then  
\[
|g_l(\cdot,t)\ast P_s|\le CBA^2s^{-\m}v(2^{-\alpha_{l-1}})2^{-\m\alpha_{l-1}}.\]
Finally we add up the estimates and take into account (\ref{eq:m}) to get
\[
|u(x,t+s)|\le CBv(s),
\]
for any $t>0$. 
 
\end{proof}

When the weight $v$ grows sufficiently fast we can reformulate the result in terms of the wavelet coefficients. For general weights such characterization is not possible (see also Example below).
  
\begin{definition}
We say that a weight $v$ is of {\it{power-type growth}} if the doubling condition (\ref{eq:double}) is fulfilled and there exists $d$ such that the sequence $\alpha_j$ defined in \ref{ss:dec} satisfies $\alpha_{j+1}-\alpha_j\le d$ for any $j\ge 0$. 
\end{definition}

Typical examples of weights of power-type growth are $v(t)=t^{-a}$, $a>0$. Normal weights in the terminology of Shields and Williams, \cite{SW1}, are of power-type growth. When $v$ is a weight of power-type growth, harmonic functions in $h_{v}^\infty$ can be characterized by their wavelet {\it{coefficients}} if one combines Corollary \ref{cor:1} with the one below.

\begin{corollary} \label{cor:2}
Let $v$ be a weight of power-type growth, and let $u$ be harmonic in $\R^{n+1}_+$ and bounded  on each half-space
$\{(x,t)\in\R^{n+1}, t\ge t_0>0\}$. Suppose there exists $B$ such that
\beq\label{eq:coef2}
|b_k(u(\cdot,t))|\le B,\quad{\text{and}}\quad
|c_{p,jk}(u(\cdot,t))|\le 2^{-nj/2}Bv(2^{-j}),
\eeq
for any $t>0$, $j\in\Z_+,\ k\in\Z^n$. Then $u\in h_{v}^\infty$ and $\|u\|_{v,\infty}\le CB$, where $C$ does not depend on $u$.
\end{corollary} 

\begin{proof}
Let $g_l$ be defined as in Theorem \ref{th:1}. We want to show that $\|g_l(x,t)\|\le Bv(2^{-\alpha_l})$. Since we have only finitely many dyadic generations between $\alpha_{l-1}$ and $\alpha_l$ it suffices to estimate 
\[
\sum_{k\in \Z^n}(u(y,t),\psi_{p,jk}(y))\psi_{p,jk}(x)\]
for each $j$. Applying (\ref{eq:coef2}) and the inequality (see \cite[ch 3.1]{Mey})
\[
\max_x\sum_ {k\in \Z^n}|\psi_{p,jk}(x)|\le C2^{nj/2},\]
we get the required estimate.
\end{proof}

\subsection{Wavelet characterization}
In this subsection we prove Theorem \ref{th:intr} formulated in the introduction. It follows readily from the proofs of Theorem \ref{th:1} and \ref{th:2}, we provide an argument below for the sake of completeness. We choose to reformulate the statement using blocks of wavelet decomposition since it is more convenient for the application we will give in the next section. 

Suppose that $u\in h_v^\infty$ and $\alpha_{l-1}\le N< \alpha_{l}$. We have
\[
s_N(u(x,t))=\int_{\R^n}E_N(x,y)u(y,t)dy,\]
and $v(2^{-N})\ge cA^l$.
Then, applying (\ref{eq:newN}) for this $l$, we obtain $|s_N(u(x,t))|\le C\|u\|_{v,\infty}v(2^{-N})$. If in addition $u\in h_v^0$, then by Corollary \ref{cor:0} we have 
\[
M_N(u)=\sup_t|s_N(u(x,t))|=o(v(2^{-N})),\quad N\rightarrow\infty.\]

To prove the converse, assume that $M_N(u)\le \varepsilon_l v(2^{-N})$ when $N\ge \alpha_{l-1}$. Then clearly
\[|g_l(x,t)|=|s_{\alpha_l}(x,t)-s_{\alpha_{l-1}}(x,t)|\le 2\varepsilon_l v(2^{-\alpha_l}).\]
Theorem \ref{th:2} implies that $u\in h_v^\infty$ when $\varepsilon_l$ are bounded.  When $\varepsilon_l\rightarrow 0$ as $l\rightarrow\infty$, we get
\[
|g_l(\cdot,t)\ast P_s|\le 2\varepsilon_l v(2^{-\alpha_l})\]
for any $t>0$. Moreover, as in the proof of Theorem \ref{th:2},
\[
|g_l(\cdot,t)\ast P_s|\le C\varepsilon_lv(2^{-\alpha_{l}})(2^{\alpha_{l-1}}s)^{-r}.\]
Then we choose $L$ such that $s\in[2^{-\alpha_{L+1}}, 2^{-\alpha_L})$ and write
\begin{multline*}
|u(x,t+s)|\le \sum_{l=0}^{L+1}2\varepsilon_lv(2^{-\alpha_l})+C\varepsilon_L\sum_{l=L+2}^\infty
v(2^{-\alpha_{l}})(2^{\alpha_{l-1}}s)^{-m}\le
\\
\sum_{l=0}^{L+1}2\varepsilon_lv(2^{-\alpha_l})+C_1\varepsilon_L(2^{\alpha_{L+1}}s)^{-m}v(2^{-\alpha_{L+1}})\le c_Lv(s),
\end{multline*}
where $c_L$ goes to zero as $L$ goes to infinity. This concludes the proof of Theorem \ref{th:intr}. 

We will also reformulate the result in terms of the boundary values of harmonic functions. Let $h^{-a}=h^\infty_{t^{-a}}$ for $a>0$ and $h^{-\infty}=\cup_a h^{-a}$. The doubling condition on the weight $v$ implies that $h_v^\infty\subset h^{-a}$ for some $a>0$. Harmonic functions in $h^{-\infty}$ admit boundary values in the sense of distributions of finite order, see for example \cite{S}. Thus when we take sufficiently smooth multiresolution approximation and choose compactly supported wavelets, we can define the wavelet coefficients of the boundary values of $u\in h^\infty_v$. Then we can reformulate the main result in the following way.

Suppose that $u\in h^{-a}$ and let $U$ be the boundary values of $u$ in the sense of distributions. Let further $b_k(U)$ and $c_{p,jk}(U)$ be the wavelet coefficients of $U$ with respect to a sufficiently smooth compactly supported wavelet basis. We define 
\[
s_N(U)(x)=\sum_{p=1}^q\sum_{j=0}^{N} \sum_{k\in\Z^n}c_{p,jk}(U)\psi_{p,jk}(x)+\sum_{k\in\Z^n} b_k(U)\phi(x-k).\]

\begin{corollary} \label{c:3}
Let $h_v^\infty\subset h^{-a}$ and let $u\in h^{-a}$. Then $u\in h_v^\infty$ if and only if there exists $K>0$ such that
\[\|s_N(U)\|_\infty\le Kv(2^{-N})\]
for any $N$.
\end{corollary}

\begin{proof}
Clearly, $c_{p,jk}(U)=\lim_{t\rightarrow 0} c_{p,jk}(u(\cdot,t))$ and then
\[
s_N(U)(x)=\lim_{t\rightarrow 0} s_N(u(\cdot,t))(x),\quad x\in\R^n.\]
Thus if $u\in h_v^\infty$ the required estimate holds.

We want to prove the converse. Consider the sequence $u_N(\cdot,y)=s_N(U)\ast P_y$ of harmonic functions in the upper half-space. By repeating the estimates from the proof of Theorem \ref{th:2}, we conclude that $u_N\in h_v^\infty$ and $\|u_N\|_{v,\infty}\le CK.$ Thus $\{u_N\}$ form a normal family in the upper half-space and we can choose a convergent subsequence $\{u_{N_j}\}$ that converges to $u_0\in h_v^\infty$. Further, let $U_0$ be the boundary values of $u_0$. We have $c_{p,jk}(U_0)=c_{p,jk}(U)$ and $b_k(U_0)=b_k(U)$. This implies $U_0=U$ and since $u_0$ and $u$ are bounded in $\{(x,t), t\ge 1\}$ we conclude that $u=u_0\in h_v^\infty$.  
\end{proof}  

In the same way Corollary \ref{cor:2} can be reformulated on the level of boundary values. We show that the corresponding result does not hold for weights that are not of power-type growth.

\begin{example} Assume that $v$ is not of power-type growth but satisfies the doubling condition. Then $\sup_l|\alpha_{l+1}-\alpha_l|=\infty$. We will give an example of a function $u\in h^{-a}$ for some $a>0$ with boundary values $U$ that satisfy
\[|c_{p,jk}(U)|\le 2^{-nj/2}v(2^{-j})\]
and $u\not\in h_v^\infty$.

Let $\psi=\psi_1$ be one of the wavelet-functions, we do not use the others so we omit the index. There exists a cube $Q_1$ in $\R^n$ where $\psi>a>0$  and for some $j$ large enough and an appropriate $k\in \Z^n$ we have $\supp\,\psi_{jk}\subset Q_1$. Iterating this construction we can find a sequence of cubes $Q_1\supset Q_2\supset...\supset Q_l\supset...$ and a sequence $k_1,...,k_l,...\in\Z^n$  such that $\supp\psi_{lj,k_l}\subset Q_l$ and $\psi_{lj,k_l}>2^{nlj/2}a$  on $Q_{l+1}$. 

Further, there exists a sequence $l_d$ such that $\alpha_{l_d+1}>\alpha_{l_d}+j(d+1)$. Then we have $\alpha_{l_d}\in[j(s_d-1), js_d)$ for some integer $s_d\ge 1$,  and we define 
\[
\nu_{d}(x)=\sum_{s=s_d}^{s_d+d} 2^{-n(sj)/2}v(2^{-js_d})\psi_{sj, k_s}(x).\]
Clearly, 
\[
\|\nu_{d}\|_\infty\ge ad\,v(2^{-js_d}).\]
Finally, let $U=\sum_d \nu_d$ then $u(\cdot,y)=U\ast P_y$ is a harmonic function in $h^{-\infty}$ but $u\not\in h_v^\infty$.  
\end{example}

\subsection{Isomorphism classes of growth spaces}
We conclude this section by one more corollary. 

\begin{corollary} Let $v$ satisfy (\ref{eq:double}). Then the Banach space $h_v^\infty$ is isomorphic to $l^\infty$.
\end{corollary}

As we mentioned in the introduction the result is known for harmonic functions in the unit disk, see \cite{L0,L00}. We suggest a new approach.  For the case of power-type growth weight the isomorphisms are straightforward, consider $T:h_v^\infty\rightarrow l_w^\infty$ given by $T(u)=(\{c_{p,jk}(U)\},\{b_k(U)\})$, where $U$ is the boundary values of $u$ in the sense of distributions. We have 
\[
l_w^\infty=\{(\{a_{p,jk}\},\{A_k\}): \max(\sup_{p,jk}|a_{p,jk}|2^{nj/2}(v(2^{-j}))^{-1}, \sup_k A_k)<+\infty\}\sim l^\infty.\]
The same operator gives isomorphism $h_v^0\rightarrow c_{w,0}$.

\begin{proof} For general weights that are not of power-type growth the proof contains two steps. First we construct an isomorphism between $h_v^\infty$  and a subspace $\ell$ of a weighted space $l_w^\infty$, then we show that $\ell$ is a complemented subspace in $l_w^\infty$.  It remains to refer to the theorem that every infinite-dimensional complemented subspace of $l^\infty$ is isomorphic to $l^\infty$, see \cite{LT}. The last step is not constructive, it exploits the Pe\l czy\'{n}ski decomposition method.

To define the isomorphism we consider the "father wavelet" $\phi$ and its dilations $\phi_{jk}(y)=2^{nj}\phi(2^{j}y-k)$, note that we choose the factor to preserve the $L^1$-norm and the integral of the function $\phi$. We assume as above that $\phi$ is smooth enough and has compact support. Now let $\A=\{\alpha_l\}_{l=0}^\infty$, we define $S: h_v^\infty\rightarrow l_w^\infty$,
\[
S(u)=\{(U, \phi_{jk})\}_{j\in\A, k\in\Z^n}
\quad {\text{and}}\quad 
\|\{a_{jk}\}\|_{\infty,w}=\sup_{jk} |a_{jk}|(v(2^{-j}))^{-1}.\]

Our results imply that $S$ is a bounded injective operator. We will describe $\ell=S(h_v^\infty)$. First, note that there are some natural connections between the coefficients $(U,\phi_{jk})$. For any $j$ we have $\phi=\phi_{00}\in V_{0}\subset V_{j}$ and 
\begin{equation}\label{eq:phirec}
\phi_{00}(y)=\sum_{m\in\Z^n} \gamma(j,m)\phi_{jm}(y).
\end{equation} 
Computing $L^2$-norms we obtain
\begin{equation}\label{eq:l2}
\sum_m\gamma(j,m)^2=2^{-nj}.
\end{equation} 
We have also $\phi_{\alpha k}=\sum_{m}\gamma(j,m-2^jk)\phi_{j+\alpha,m}$ and since $\{\phi_{jm}\}$ is a basis in $V_j$, we have
\[
\sum_m\gamma(j,m-2^j k)\gamma(i,\mu-2^i m)=\gamma(i+j,\mu-2^{i+j}k).\]

Moreover, the orthogonality of $\phi_{\alpha k}$ and $\phi_{\alpha \kappa}$ gives 
\[
\sum_m\gamma(j, m-2^jk)\gamma(j, m-2^j\kappa)=0,\quad k\neq \kappa.\] 
Finally, the last three identities imply
\[
\sum_\mu\gamma(i,\mu-2^im)\gamma(i+j,\mu-2^{i+j}k)=2^{-ni}\gamma(j, m-2^jk).\]

Now we define
\begin{multline*}
\ell=\{\{a_{jk}\}_{j\in \A, k\in\Z^n}\in l_w^\infty: a_{jk}=\sum_{m\in\Z^n}\gamma(j'-j,m-2^{j'-j}k) a_{j' m},\\
{\text{ for all}}\ j,j'\in\A, j<j', k\in\Z^n\}.
\end{multline*}
To be more precise we could right $l_w^\infty(\A\times\Z^n)$, meaning that we consider the space of all bounded functions of the countable set $\A\times\Z^n$,  but we avoid this complication of notation and hope it does not cause misinterpretation.

It is not difficult to prove that for $\{a_{jk}\}\in\ell$ the following limit 
\[
\lim_{j\rightarrow\infty}\sum_k 2^{-nj} a_{jk}\phi_{jk}\]
exists in the sense of distributions and the limit is the boundary values of a function in $h_v^\infty$ (see the proof of Corollary \ref{c:3} above). This gives the inverse of $S$ restricted to $\ell$. 

To complete the proof we give a bounded projection from $l_w^\infty$ to $\ell$. Let $\{a_{jk}\}\in l_w^\infty$, for each $j\in\A$ and $k\in\Z^n$ we define 
\[\delta_{jk}=a_{jk}-\sum_{m\in\Z^n}\gamma(j'-j,m-2^{j'-j}k) a_{j' m},\]
where $j'$ is the next element in $\A$ after $j$.
Let further
\[\tilde{a}_{jk}=a_{jk}+\sum_{i<j, i\in\A}2^{n(j-i)}\sum_{m\in \Z^n}\gamma(j-i,k-2^{j-i}m)\delta_{i m}.\]
Straightforward but tedious calculations show that $\{a_{jk}\}\mapsto\{\tilde{a}_{jk}\}$ is a projection on $\ell$. To show that it is bounded we note that the number of non-zero terms in (\ref{eq:phirec}) does not exceed $C2^{nj}$ and use the estimate (\ref{eq:l2}). It implies that $\|\{\delta_{jk}\}\|_{\infty,w}\le C\|\{a_{jk}\}\|_{\infty,w}$ and further by (\ref{eq:m}) we get $\|\{\tilde{a}_{jk}\}\|_{\infty,w}\le C\|\{a_{jk}\}\|_{\infty,w}$. This completes the proof of  the corollary.

\end{proof}

\section{Oscillation along vertical lines}\label{s:osc}
\subsection{An auxiliary estimate} Let $u\in h_{v}^\infty$, we decompose it into wavelet blocks like in Theorem \ref{th:1} and prove one more inequality for the blocks $g_l(x,t)$ when $t>2^{-\alpha_l}$. It shows how the smoothness of the multiresolution analysis affects the error of the approximation. 

\begin{proposition}
Let $u\in h_v^\infty(\R^{n+1}_+)$ and let $g_l$ be defined as in Theorem \ref{th:1} and let $m$ be a positive integer, $m\le r$. Then there exists $C=C_m$ such that 
\beq\label{eq:waved}
\|g_l(\cdot,t)\|_\infty\le C\|u\|_{v,\infty}v(2^{-\alpha_l})(2^{\alpha_{l-1}}t)^{-m},\quad l\ge 0,
\eeq
for any $t>2^{-\alpha_{l-1}}$.
\end{proposition}

\begin{proof}
As above we have
\[
g_l(x,t)=\int_{\R^n}(E_{\alpha_l}(x,y)-E_{\alpha_{l-1}}(x,y))u(y,t)dy.\]
Now $u(y,t)=(u(\cdot,t/2)\ast P_{t/2})(y)$ and 
\[
|g_l(x,t)|\le\|u\|_{v,\infty}v(t/2)
\int_{\R^n}
\left|\int_{\R^n}(E_{\alpha_l}(x,y)-E_{\alpha_{l-1}}(x,y))P_{t/2}(y-w)dy\right|dw.\]
Rescaling and applying Lemma \ref{l:DP} we get (\ref{eq:waved}).
\end{proof}


\subsection{Vertical weighted average} Now we consider the following integral
\[
I_u(x,s)=\int_{s}^1u(x,t)d(1/v(t))\]
and its multiresolution approximations.
We use the wavelet decomposition of $u(x,t)$ and write $I_u$ as the sum of functions $G_l(x,s)$,
\[
G_l(x,s)=\int_{s}^1g_l(x,t)d(1/v(t)),\quad l\ge0,
\]
where $g_l$ were defined in Theorem \ref{th:1}.

\begin{lemma}
Let $\m$ satisfy (\ref{eq:m}), $r>\m+n$, and let $G_l(x,s)$ be defined as above. Then
\beq\label{eq:Gs}
|G_l(x,s)|\le C \|u\|_{v,\infty}2^{-\m\alpha_{l-1}}s^{-\m}v(2^{-\alpha_l})v(s)^{-1},
\eeq 
when  $s\ge 2^{-\alpha_{l-1}}$ and
\beq\label{eq:G}
|G_l(x,s)|\le C\|u\|_{v,\infty},\quad s>0.
\eeq
\end{lemma}
\begin{proof}
Let $s\ge 2^{-\alpha_{l-1}}$. There exists $L<l-1$ such that $s\in(2^{-\alpha_{L+1}},2^{-\alpha_{L}})$. Applying (\ref{eq:waved}) we obtain
\[
|G_l(x,s)|\le C\|u\|_{v,\infty}v(2^{-\alpha_l})\int_s^1 2^{-\m\alpha_{l-1}}t^{-\m}d(1/v(t)).\]
Then integrating over the intervals $(2^{-\alpha_{i}},2^{-\alpha_{i-1}})$ and using (\ref{eq:m}), we have
\begin{multline*}
\int_s^1 t^{-\m}d(1/v(t))=\int_s^{2^{-\alpha_{L}}}t^{-\m}d(1/v(t))+\sum_{i=1}^{L}\int_{2^{-\alpha_{i}}}^{2^{-\alpha_{i-1}}}t^{-\m}d(1/v(t))
\le \\
A^2\left(\frac{s^{-\m}}{v(s)}+\sum_{i=1}^{L-1}\frac{2^{\m\alpha_i}}{v(2^{-\alpha_i})}\right)\le 
C_1\left(\frac{s^{-\m}}{v(s)}+\frac{2^{\m\alpha_{L}}}{v(2^{-\alpha_{L}})}\right)\le
C_2s^{-\m}v(s)^{-1}.
\end{multline*}
Inequality (\ref{eq:Gs}) follows.

Further, for $s<2^{-\alpha_{l-1}}$ we apply (\ref{eq:wave}) and get
\[
|G_l(x,s)|\le |G_l(x,2^{-\alpha_{l-1}})|+\int_s^{2^{-\alpha_{l-1}}}|g_l(x,t)|d(1/v(t))\le
C\|u\|_{v,\infty}.
\]
\end{proof}
The lemma implies that
\[
G_l(x)=\int_0^1g_l(x,t)d(1/v(t))\]
is well-defined, $|G_l(x)|\le C\|u\|_{v,\infty}$ and $G_l\in V_{\alpha_l}(\infty)$.


\subsection{Martingale approximation}
We will approximate $I_u(x,2^{-\alpha_L})$ by a martingale. First we remark that for $s\in (2^{-\alpha_{L+1}},2^{-\alpha_{L}})$
\beq\label{eq:Ia}
|I_u(x,s)-I_u(x, 2^{-\alpha_{L}})|\le
\int_s^{2^{-\alpha_L}}|u(x,t)|d(1/v(t))\le C\|u\|_{v,\infty}.\eeq
Now let $\F_L$ be the $\sigma$-algebra generated by dyadic cubes of size $2^{-\alpha_L}$. We define
\[
\Lambda_{l,L}=\cE(G_l|\F_L).\]
Clearly $|\Lambda_{l,L}|\le C\|u\|_{v,\infty}$ and for each $l\ge 0$ we obtain a martingale $\{\Lambda_{l,L}\}_{L=1}^\infty$. 

Now we assume that the wavelets have compact supports and remind that
\[
G_l(x)=\sum_{j=\alpha_{l-1}+1}^{\alpha_l}\sum_{p=1}^q\sum_{k\in\Z^n}\int_0^1(u(y,t),\psi_{p,jk}(y))d(1/v(t))\psi_{p,jk}(x),\quad  l\ge 1.\]

\begin{lemma}\label{l:Lambda}
Let $\Lambda_{l,L}$ be as above, assume that $L<l$ then
\[
|\Lambda_{l,L}(x)|\le C\|u\|_{v,\infty} 2^{\alpha_L-\alpha_{l-1}}
\]
for any $x\in\R^n$. 
\end{lemma}
\begin{proof}
We have 
\[
G_l(x)=\sum_{j=\alpha_{l-1}+1}^{\alpha_l}\sum_{p=1}^q\sum_{k\in\Z^n} a_{p}^{jk}\psi_{p,jk}(x),\]
where the series converges uniformly in $x$ and 
\[|a_p^{jk}|\le \|G_l\|_\infty\|\psi_{p,jk}\|_1\le C2^{-jn/2}\|u\|_{v,\infty}.
\]
Let $x\in Q$, where $Q$ is a dyadic cube of size $2^{-\alpha_L}$. We fix $l>L$ and consider wavelets $\psi_{p,jk}$ with $\alpha_{l-1}<j\le \alpha_l$ such that $\supp(\psi_{p,jk})\cap Q\neq\emptyset$ but $\supp(\psi_{p,jk})\not\subset Q$. In each dyadic generation $j$ we have at most $C2^{(n-1)(j-\alpha_L)}$ such wavelets and the $L^1$-norm of wavelets in this generation is bounded by $C2^{-jn/2}$.  Clearly
\begin{multline*}
|\Lambda_{l,L}(x)|=|\cE(G_l|\F_L)(x)|=\left| \sum_{j=\alpha_{l-1}+1}^{\alpha_l}\sum_{p=1}^q\sum_{k\in\Z^n} a_{p}^{jk}|Q|^{-1}\int_Q\psi_{p,jk}(y)dy\right|\le\\
C\|u\|_{v,\infty}\sum_{j=\alpha_{l-1}+1}^{\alpha_l}2^{(n-1)(j-\alpha_L)}2^{-jn/2}2^{n\alpha_L}2^{-jn/2}\le C\|u\|_{v,\infty} 2^{\alpha_L-\alpha_{l-1}}.
\end{multline*}
\end{proof}

Then the series $\Gamma_L(x)=\sum_{l=0}^\infty\Lambda_{l,L}(x)$ converges uniformly on $\R^n$ and $\{\Gamma_L\}_{L=0}^\infty$ is a martingale with respect to the sequence of $\sigma$-algebras $\F_L$.
It turns out that this martingale provides a good approximation for the weighted average $I_u$.

\begin{lemma}\label{l:MA} 
Let $u\in h_{v,\infty}$ and let $I_u$ and $\Gamma_L$ be defined as above, then
\beq\label{eq:MA}
|I_u(x,2^{-\alpha_L})-\Gamma_L(x)|\le C\|u\|_{v,\infty}.
\eeq
\end{lemma}
\begin{proof}
First, by (\ref{eq:Gs}) and (\ref{eq:m}) we have
\[
|I_u(x,2^{-\alpha_L})-\sum_{l=0}^LG_l(x,2^{-\alpha_L})|\le C\|u\|_{v,\infty}.\]
Similarly, by Lemma \ref{l:Lambda},
\[
|\Gamma_L(x)-\sum_{l=0}^L\Lambda_{l,L}(x)|\le C\|u\|_{v,\infty}.\] 
Further, for $l\le L$, using (\ref{eq:wave}) we get
\begin{multline*}
|\Lambda_{l,L}(x)-G_l(x,2^{-\alpha_L})|\le
|\Lambda_{l,L}(x)-G_l(x)|+
|G_{l}(x)-G_l(x,2^{-\alpha_L})|\le\\
\|\nabla G_l\|_\infty 2^{-\alpha_L}+C\|u\|_{v,\infty}v(2^{-\alpha_l})(v(2^{-\alpha_L}))^{-1}\le
C\|u\|_{v,\infty}(2^{\alpha_l-\alpha_L}+A^{l-L}).
\end{multline*}
In the last step we applied the Bernstein inequality (\ref{eq:Bern}) to $G_l\in V_{\alpha_l}(\infty)$. Finally, we sum up the estimates to obtain
(\ref{eq:MA}).
\end{proof}

Now we can prove the law of the iterated logarithm for the weighted average.
\begin{theorem}\label{th:osc}
Let $u\in h_{v,\infty}$ and
\[
I_u(x,s)=\int_s^1 u(x,t)d\left(\frac1{v}(t)\right).\]
Then 
\begin{equation}\label{eq:osc}
\limsup_{s\rightarrow 0+}\frac{I_u(x,s)}{\sqrt{\log v(s)\log\log\log v(s)}}\le C\|u\|_{v,\infty},
\end{equation}
for almost every $x\in\R^n$.
\end{theorem}
\begin{proof}
We consider the martingale $\{\Gamma_L\}_{L=0}^\infty$ and estimate its square function 
\[S^2_N(x)=\sum_{L=0}^N\cE(|\Gamma_{L+1}-\Gamma_L|^2|\F_L).\] 
We have
\[
|\Gamma_{L+1}(x)-\Gamma_{L}(x)|\le \sum_{l=0}^L|\Lambda_{l,L+1}(x)-\Lambda_{l,L}(x)|+
\sum_{l=L+1}^\infty\left(|\Lambda_{l,L+1}(x)|+|\Lambda_{l,L}(x)|\right).\]
To estimate the first sum we use the Bernstein inequality as in the proof of Lemma \ref{l:MA}, for the second sum we apply the inequality of Lemma \ref{l:Lambda}. Then we get
\begin{multline*}
|\Gamma_{L+1}(x)-\Gamma_{L}(x)|\le\\ \sum_{l=0}^L\|\nabla G_l\|_{\infty}2^{-\alpha_L}+|\Lambda_{L+1,L+1}(x)|+
C\|u\|_{v,\infty}\sum_{l=L+1}^\infty 2^{\alpha_{L+1}-\alpha_l} \le C\|u\|_{v,\infty}.
\end{multline*}
Thus we have a martingale with bounded differences and also $S^2_N(x)\le C\|u\|^2_{v,\infty}N$. Now, we note that $\log v(2^{-\alpha_N})\ge CN$,  
 combine inequalities (\ref{eq:Ia}), (\ref{eq:MA}) and apply the law of the iterated logarithm for martingales, see for example \cite{Stbook}. 
\end{proof}

\section{Concluding remarks and open problems}

It would be interesting to consider growth spaces with weights that depend on $x\in\R^n$ as well. It seems that wavelets should be better adjusted to local behavior of the functions than for example Fourier series and transforms. We expect that the following local version of Theorem \ref{th:osc} holds. If $E$ is a set of positive measure such that $|u(x,t)|\le Cv(t)$ when $x\in E$ then (\ref{eq:osc}) holds for almost all $x\in E$. Another related question is to generalize Theorem \ref{th:osc} to Lipschitz domains as it is done for Bloch functions and square function inequalities for harmonic functions in \cite{Ll} and \cite{R}.

The transition from the smooth wavelet representation into a martingale (that is Haar-wavelet representation) done in the last section, does not look very natural. Supposedly it can be omitted by proving the law of the iterated logarithm for wavelet series directly, see also \cite{C}. However, with the tools developed we find it easier to refer to well-known results for martingales. 

The following remark is due to J. G. Llorente and A. Nicolau.
\begin{remark}  
Let us note that $I_u(x,s)$ behaves like the following harmonic function
\[
H_u(x,s)=\int_0^1 u(x,t+s)d(1/v(t)).\]
The classical iterated logarithm theorem for harmonic functions gives an estimate for $H_u$ in terms of its area function. It is standard that
\[
|\nabla u(x,s)|\le C v(s)s^{-1}\] and a straightforward estimate gives
\[
|\nabla H_u(x,s)|\le  \int_0^1v(t+s)(t+s)^{-1}d(1/v(t))\lesssim s^{-1}.
\]
Let $\Gamma(x,h)$ be the doubly truncated cone, 
\[\Gamma(x,t)=\{(y,s)\in\R^{n+1}_+, |x-y|<s,\ t<s<1\},\] the square function $A(H)$ of a function $H$ is defined by
\[
A(H)(x,t)=\left(\int_{\Gamma(x,h)}|\nabla H(y,s)|^2 s^{1-n}dyds\right)^{1/2}.\]
 Then for a harmonic function $H$ one has (see \cite[Theorem 3.0.4]{BM})
\[
\limsup_{t\rightarrow 0}\frac{|H(x,t)|}{A(H)(x,t/2)\sqrt{\log\log(A(H)(x,t/2))}}\le C,\]
for almost every $x\in\{x\in \R^n: \lim_{h\rightarrow 0} A(H)(x,h)=\infty\}$. (Almost everywhere on the complement of this set $H$ has non-tangential limits.)
For the function $H_u$ we have $A(H_u)(x,t)\le C\sqrt{|\log t|}.$
This implies 
\[
\limsup_{s\rightarrow 0+}\frac{H_u(x,s)}{\sqrt{\log s^{-1}\log\log\log s^{-1}}}\le C,\]
for almost every $x\in\R^n$. For slowly growing $v$ we suggest a better estimate, we cannot derive it from the estimates of the area function of $H$. We know that this area function does not satisfy $A(H_u)(x,t)=O(\log v(t))$, but it seems like we should have $\lim_{t\rightarrow 0} A(H_u)(x,t)(\log v(t))^{-1}<+\infty$ a.e. 
\end{remark}

 If we consider $H_u(x,s)$ defined above, then it can be rewritten as a multiplier on the Fourier transform of the boundary values of $u$. This operator is similar to multipliers considered by Shields and Williams, we refer in particular to Theorem 5 in \cite{SW}. Our multiplier is quite interesting since its growth matches exactly the growth of the weight. Then the growth of the resulting function is difficult to catch precisely in the scale of growth spaces, instead we obtain the law of the iterated logarithm that gives accurate asymptotic estimate almost everywhere. It would be interesting to place the result on the oscillation integral in a more general context of multipliers. 

\section*{Acknowledgments}
The authors are grateful to Yurii Lyubarskii and Artur Nicolau for fruitful discussions of the problem.  It is our pleasure to thank the referee who  read the manuscript carefully and made useful suggestions.

Part of the work was done while P. Mozolyako was visiting the Department of Mathematical Sciences at NTNU and  E. Malinnikova was visiting the Chebyshev Laboratory at St.~Petersburg State University, we would like to thank both  Universities for hospitality and great working conditions.     

This work was supported by the 
Research Council of Norway, grants 185359/V30 and 213638, by the Chebyshev Laboratory (Department of Mathematics and Mechanics, St.~Petersburg State University) under RF government grant 11.G34.31.0026, and by RFBR grant 12-01-31492.

\end{document}